\documentclass[12pt, reqno]{amsart}
 \usepackage{amsmath, amsthm, amscd, amsfonts, amssymb, graphicx, color, float}
\usepackage[bookmarksnumbered, colorlinks, plainpages]{hyperref}
\input{mathrsfs.sty}
\hypersetup{colorlinks=true,linkcolor=red, anchorcolor=green, citecolor=cyan, urlcolor=red, filecolor=magenta, pdftoolbar=true}

\newtheorem{theorem}{Theorem}[section]
\newtheorem{lemma}[theorem]{Lemma}

\newtheorem{corollary}[theorem]{Corollary}
\theoremstyle{definition}

\theoremstyle{remark}
\newtheorem{remark}[theorem]{Remark}
\numberwithin{equation}{section}
\begin{document}

\title[Inequalities in noncommutative probability spaces]{Etemadi and Kolmogorov inequalities in noncommutative probability spaces}
\author[A. Talebi, M.S. Moslehian, Gh. Sadeghi]{Ali Talebi$^1$, Mohammad Sal Moslehian $^1$$^*$ and Ghadir Sadeghi$^2$}

\address{$^1$ Department of Pure Mathematics, Center of Excellence in Analysis on Algebraic Structures (CEAAS), Ferdowsi University of
Mashhad, P.O. Box 1159, Mashhad 91775, Iran}
\email{alitalebimath@yahoo.com }
\email{moslehian@um.ac.ir, moslehian@member.ams.org}

\address{$^2$ Department of Mathematics and Computer Sciences, Hakim Sabzevari University, P.O. Box 397, Sabzevar, Iran}
\email{ghadir54@gmail.com, g.sadeghi@hsu.ac.ir}

\thanks{$^*$ Corresponding author}


\begin{abstract}
Based on a maximal inequality type result of Cuculescu, we establish some noncommutative maximal inequalities such as Haj\'ek--Penyi inequality and Etemadi inequality. In addition, we present a noncommutative Kolmogorov type inequality by showing that if $x_1, x_2, \ldots, x_n$ are successively independent self-adjoint random variables in a noncommutative probability space $(\mathfrak{M}, \tau)$ such that $\tau\left(x_k\right) = 0$ and $s_k s_{k-1} = s_{k-1} s_k$, where $s_k = \sum_{j=1}^k x_j$, then for any $\lambda > 0$ there exists a projection $e$ such that
$$1 - \frac{(\lambda + \max_{1 \leq k \leq n} \|x_k\|)^2}{\sum_{k=1}^n {\rm var}(x_k)}\leq \tau(e)\leq \frac{\tau(s_n^2)}{\lambda^2}.$$
As a result, we investigate the relation between convergence of a series of independent random variables and the corresponding series of their variances.
\end{abstract}

\maketitle

\section{Introduction and preliminaries}

One of important types of probability inequalities relates tail probabilities for the maximal partial sum of independent random variables such as Kolmogorov and Etemadi inequalities. It is known that a strong law of large numbers can be deduced from a maximal inequality for a certain sequence of random variables \cite{BAT}.
In \cite{HR}, Haj\'ek and Renyi extended the maximal Kolmogorov inequality to weighted sums. It asserts that if $X_1, X_2, \ldots, X_n$ are independent mean zero random variables and $\{\alpha_k: 0\leq k \leq n \}$ is a non-increasing positive real numbers, then
\begin{align*}
\mathbb{P}\left(\max_{1 \leq k \leq n} \alpha_k\left|\sum_{i=1}^k X_i\right| \geq t \right) \leq \frac{\sum_{k=1}^n\alpha_k^2{\rm Var}(X_k)}{t^2}
\end{align*}
for any $t > 0$.
In \cite{E}, Etemadi established an inequality, which deduces a weak law of large numbers. It states that if $X_1, X_2, \ldots, X_n$ are independent random variables defined on a probability space $(\Omega, \mathcal{F}, \mathbb{P})$, then
\begin{align*}
\mathbb{P}\left(\max_{1 \leq k \leq n} \left|\sum_{i=1}^k X_i\right| \geq 3t \right) \leq 3 \max_{1 \leq k \leq n} \mathbb{P}\left(\left|\sum_{i=1}^k X_i\right| \geq t \right)
\end{align*}
for any positive real number $t>0$.

One of the major problems occurring in the noncommutative probability theory concerns with the extensions of classical inequalities to the noncommutative setup; cf \cite{SM1, SM2, TMS}. In the sequel, we recall some basic preliminaries concerning noncommutative probability spaces that will be needed throughout.

A von Neumann algebra $\mathfrak{M}$ on a Hilbert space $\mathcal{H}$ with the unit element $\textbf{1}$ equipped with a normal faithful tracial state $\tau:\mathfrak{M}\to \mathbb{C}$ is called a noncommutative probability space. The elements of $\mathfrak{M}$ are called (non-commutative) random variables. We denote by $\leq$ the usual order on the self-adjoint part $\mathfrak{M}^{sa}$ of $\mathfrak{M}$.
Recall that $L_p(\mathfrak{M})$ can be regarded as the subspace of the $p$-integrable operators in the algebra $\overline{\mathfrak{M}}$ of the measurable operators in the sense of Nelson (see \cite{N}).\\
For each self-adjoint operator $x\in \overline{\mathfrak{M}}$, there exists a unique spectral measure $E$ as a $\sigma$-additive mapping with respect to the strong operator topology from the Borel $\sigma$-algebra $\mathcal{B}(\mathbb{R})$ of $\mathbb{R}$ into the set of all orthogonal projections such that for every Borel function $f: \sigma(x)\to \mathbb{C}$ the operator $f(x)$ is defined by $f(x)=\int f(\lambda)dE(\lambda)$, in particular, $\textbf{1}_B(x)=\int_BdE(\lambda)=E(B)$.
Further, if $x\geq 0$ and $t>0$ then the inequality
\begin{eqnarray*}
{\rm Prob}(x\geq t):=\tau(\textbf{1}_{[t,\infty)}(x))\leq t^{-1}\tau(x)\,.
\end{eqnarray*}
is known as the Markov inequality in the literature. For any self-adjoint element $x \in L_2(\mathfrak{M})$ and $\lambda>0$, it follows from the Markov inequality that
\begin{eqnarray}\label{GH1}
\tau\left(\textbf{1}_{[\lambda, \infty)}(|x - \tau(x)|)\right) \leq \lambda^{-2}{\rm var}(x),
\end{eqnarray}
where ${\rm var}(x) = \tau\left( (x - \tau(x))^2\right)$ and $|x| = (x^*x)^{\frac{1}{2}}$ denotes the absolute value of $x$. Inequality (\ref{GH1}) is called the Chebyshev inequality.

Let $\mathfrak{N}$ be a von Neumann subalgebra of $\mathfrak{M}$. Then there exists a normal contractive positive mapping projecting $\mathcal{E}_{\mathfrak{N}}:\mathfrak{M}\to\mathfrak{N}$ satisfying the following properties:\\
(i) $\mathcal{E}_{\mathfrak{N}}(axb) =a\mathcal{E}_{\mathfrak{N}}(x)b$ for any $x\in\mathfrak{M}$ and $a, b\in\mathfrak{N}$;\\
(ii) $\tau\circ\mathcal{E}_{\mathfrak{N}}=\tau$.\\
Moreover, $\mathcal{E}_{\mathfrak{N}}$ is the unique mapping satisfying (i) and (ii). The mapping $\mathcal{E}_{\mathfrak{N}}$ is called the conditional expectation of $\mathfrak{M}$ with respect to $\mathfrak{N}$.
Let $\mathcal{P}$ be the lattice of projections of $\mathfrak{M}$. Set $p^{\perp}=1-p$ for $p\in \mathcal{P}$. Given a family of projections $(p_{\lambda})_{\lambda\in\Lambda}\subseteq\mathcal{P}$, we denote by $\vee_{\lambda\in\Lambda}p_{\lambda}$ (resp., $\wedge_{\lambda\in\Lambda}p_{\lambda}$) the projection from $\mathfrak{H}$ onto the closed subspace generated by $\cup_{\lambda \in \Lambda}p_{\lambda}(\mathfrak{H})$ (resp., onto the subspace $\cap_{\lambda\in\Lambda}p(\mathfrak{H})$). If $(p_{\lambda})_{\lambda\in\Lambda}$ is a family of projections in $\mathfrak{M}$ then
\begin{eqnarray}\label{NE}
\tau\left(\vee_{\lambda\in\Lambda}p_{\lambda}\right)\leq\sum_{\lambda\in\Lambda}\tau(p_{\lambda}).
\end{eqnarray}
The reader is referred to \cite{Xu2} for more information.

In this paper, we adopt some methods in classical probability spaces (see e.g. \cite{G}) to the non-commutative setting. Our approaches, however, employ several essential operator algebra techniques. In fact, it is not routine to find noncommutative versions of classical inequalities involving maximum of random variables. To illustrate the difficulties, we notify that the maximum of two self-adjoint operators does not exist, in general. Further, the usual triangle inequality is no longer valid for the modulus of operators, namely, we do not have $|x + y| \leq |x| + |y|$ for $x, y \in \mathcal{M}$, in general.
Employing Cuculescu's approach we establish some noncommutative maximal inequalities such as Etemadi inequality and Haj\'ek--Renyi inequality, which is an extension of a result of \cite[Proposition 5.1]{BAT}. We also present a noncommutative Kolmogorov type inequality. We also state a relation between convergence of a series of independent random variables and the corresponding series of their variances. The commutative versions of the given inequalities are discussed as well.


\section{Noncommutative Kolmogorov type inequalities}
Recall that Kolmogorov's inequality is significantly stronger than Chebyshev's inequality.
In this section we first provide a noncommutative Haj\'ek--Renyi inequality and as a consequence, we present a noncommutative Kolmogorov inequality. After that a noncommutative Kolmogorov type inequality is given.\\
As an extension of the classical independent random variables, recall that two von Neumann subalgebras $\mathfrak{N}_1, \mathfrak{N}_2 \subset \mathfrak{M}$ are independent (over $\mathbb{C}$) if
\begin{equation*}
\tau\left( xy \right) = \tau(x)\tau(y), \quad x\in \mathfrak{N}_1, y \in \mathfrak{N}_2.
\end{equation*}
A sequence $\left(x_n\right)_{n=1}^\infty$ in $L_p(\mathfrak{M})$ is said to be successively independent \cite{JX} if the subalgebras $W^*(x_j)$ and $W^*\left( x_1, \ldots, x_{j-1}\right)$ are independent (over $\mathbb{C}$) for any $1 < j$, where $W^*(A)$ denotes the von Neumann algebra generated by the spectral projections of the real and imaginary parts of any member of $A \subset \mathfrak{M}$.\\
We introduce the other types of the notion of independence which will be used in the sequel.
We say a sequence $\left(x_k\right)_{k=1}^n$ is said to be weakly fully independent (which is weaker than fully independent defined in \cite{JZ}) if for any $1 < j < n$ the subalgebras $W^*\left( x_1, \ldots, x_{j-1}\right)$ and $W^*\left(x_{j}, \ldots, x_n\right)$ are independent.
Clearly
\begin{align*}
\text{weakly full independence} \Longrightarrow \text{successively independence}.
\end{align*}
Note that if $x, y$ are successively independent random variables, then ${\rm var}(x + y) = {\rm var}(x) + {\rm var}(y)$, because
\begin{eqnarray}\label{eqvar1}
{\rm var}(x + y) &=& \tau\left((x + y)^2\right) - \left(\tau(x + y)\right)^2 \nonumber\\
&=& \tau\left(x^2 \right) + \tau\left(y^2 \right) + 2\tau\left(xy \right) - \tau\left(x \right)^2 - \tau\left(y \right)^2 - 2\tau\left(x \right)\tau\left(y \right)\nonumber\\
&=& {\rm var}(x) + {\rm var}(y)\quad \left(\text{as $\tau\left( xy \right) = \tau(x)\tau(y)$} \right).
\end{eqnarray}
\begin{lemma}\label{lem2}
If $p ,q$ are two projections in $\mathfrak{M}$ with $p \leq q$ and $x \geq 0$, then $\| pxp \| \leq \| qxq \|$.
\end{lemma}
\begin{proof}
\begin{eqnarray*}
\| pxp \| = \| x^{\frac{1}{2}} p\|^2 = \| px^{\frac{1}{2}} \|^2 = \|x^{\frac{1}{2}} px^{\frac{1}{2}} \| \leq \|x^{\frac{1}{2}} qx^{\frac{1}{2}} \| = \| qxq\|.
\end{eqnarray*}
\end{proof}
\begin{theorem}[Noncommutative Haj\'ek--Renyi inequality]
Let $x_1, x_2, \ldots$ be successively independent self-adjoint random variables in $L_2(\mathfrak{M})$ such that $\tau\left(x_n\right) = 0$ for all $n$. Then for every decreasing sequence of positive real numbers $\{\alpha_n\}_{n=0}^\infty$ and for any $\lambda>0$ there exists a projection $p$ such that
\begin{align*}
&\tau(p) \leq \frac{\sum_{n=1}^\infty\alpha_n^2{\rm var}(x_n)}{\lambda^2},\\
&\left\| (1-p)\,\alpha_n|\sum_{k=1}^n x_k|\,(1-p) \right\| \leq \lambda \qquad (n \geq 1).
\end{align*}
\end{theorem}
\begin{proof}
Fix $n \in \mathbb{N}$ and set $s_n = \sum_{k=1}^n x_k$. We use some ideas of Cuculescu \cite[Proposition 5]{C}.
Let $e_0 = 1$ and inductively on $1\leq k\leq n$, define
\begin{align*}
e_k = \textbf{1}_{[0, \lambda)}\left( e_{k-1} \alpha_k |s_k| e_{k-1} \right)e_{k-1}.
\end{align*}
Since $e_{k-1}$ is a projection, $e_{k-1}$ commutes with $e_{k-1} |s_k| e_{k-1}$ and hence with all its spectral projections; in particular $e_{k-1}$ commutes with $\textbf{1}_{[0, 3 \lambda)}\left( e_{k-1} |s_k| e_{k-1} \right)$. So $e_k$ is also a projection in $W^*\left(x_1, x_2, \ldots, x_k\right)$. It follows from $e_k e_{k-1} = e_k$ that $\left(e_k \right)_{k=1}^n$ is a decreasing sequence.\\
We set
\begin{align*}
p_k = e_{k-1} - e_k
\end{align*}
for each $1 \leq k \leq n$.
Then $ (p_k)_k$ is a sequence of projections.\\
Note that $\tau\left(s_k^2p_i \right) \geq \tau\left(s_i^2p_i \right)$ for $k \geq i$. In fact, by the successive independence of $x_j$ 's, we have
\begin{eqnarray}\label{equ2}
\tau\left(s_k^2p_i \right) &=& \tau\left(\left(s_i + \sum_{l=i+1}^k x_l \right)^2p_i \right)\nonumber\\
&=& \tau\left(s_i^2p_i \right) + \tau\left(\sum_{l=i+1}^k x_l s_ip_i \right) + \tau\left(\sum_{l=i+1}^k x_l p_is_i \right) + \tau\left(\left(\sum_{l=i+1}^k x_l\right)^2 p_i \right)\nonumber\\
&\geq & \tau\left(s_i^2p_i \right) + \tau\left(\sum_{l=i+1}^k x_l\right) \tau\left(s_ip_i \right) + \tau\left(\sum_{l=i+1}^k x_l\right) \tau\left(p_is_i \right)\nonumber\\
&=& \tau\left(s_i^2p_i \right) \quad (\text{by $\tau(x_l)=0$ for all $l$}).
\end{eqnarray}
Therefore,
\begin{align*}
&\sum_{k=1}^n \alpha_k^2 {\rm var}\left(x_k \right)\\
&= \sum_{k=1}^n \left[\alpha_k^2\left( {\rm var}\left(s_k \right) - {\rm var}\left(s_{k-1} \right)\right)\right]\quad \left(\text{by \eqref{eqvar1}} \right)\\
&= \sum_{k=1}^{n-1} \left[\left(\alpha_k^2 - \alpha_{k+1}^2 \right){\rm var}\left(s_k \right)\right] + \alpha_n^2 {\rm var}\left(s_n \right)\\
&= \sum_{k=1}^{n-1} \left[\left(\alpha_k^2 - \alpha_{k+1}^2 \right)\tau\left(s_k^2 \right)\right] + \alpha_n^2 \tau\left(s_n^2 \right) \quad \left(\text{as $\tau\left(x_k \right)=0$}\right)\\
&\geq \sum_{k=1}^{n-1} \left[\left(\alpha_k^2 - \alpha_{k+1}^2 \right)\tau\left(s_k\left(\sum_{i=1}^np_i\right) s_k \right)\right] + \alpha_n^2 \tau\left(s_n\left(\sum_{i=1}^np_i\right) s_n \right)
\end{align*}
as $p_k p_j=(e_{j-1} - e_j)(e_{k-1} - e_k)= 0$ when $j\neq k$ and $\sum_{i=1}^np_i=1-e_n\leq 1$. So
\begin{align*}
\sum_{k=1}^n \alpha_k^2 {\rm var}\left(x_k \right)&\geq \sum_{i=1}^n\sum_{k=1}^{n-1} \left[\left(\alpha_k^2 - \alpha_{k+1}^2 \right)\tau\left(s_k^2p_i \right)\right] + \alpha_n^2\sum_{i=1}^n \tau\left(s_n^2p_i \right)\\
&\geq \sum_{i=1}^n\sum_{k=i}^{n-1} \left[\left(\alpha_k^2 - \alpha_{k+1}^2 \right)\tau\left(s_k^2p_i \right)\right] + \alpha_n^2\sum_{i=1}^n \tau\left(s_n^2p_i \right),
\end{align*}
since $\{\alpha_k\}_{k}$ is non-increasing and $\tau\left(s_k^2p_i \right) = \tau\left(p_is_k^2p_i \right) \geq 0$. Thus,
\begin{align}\label{in5}
&\sum_{k=1}^n \alpha_k^2 {\rm var}\left(x_k \right)\nonumber\\
&\geq \sum_{i=1}^n\sum_{k=i}^{n-1} \left[\left(\alpha_k^2 - \alpha_{k+1}^2 \right)\tau\left(s_i^2p_i \right)\right] + \alpha_n^2\sum_{i=1}^n \tau\left(s_i^2p_i \right) \quad \left(\text{by \eqref{equ2}} \right)\nonumber\\
&\geq \sum_{i=1}^n\sum_{k=i}^{n-1} \left[\left(\alpha_k^2 - \alpha_{k+1}^2 \right)\frac{\lambda^2}{\alpha_i^2} \tau\left(p_i \right)\right] + \alpha_n^2\sum_{i=1}^n \frac{\lambda^2}{\alpha_i^2} \tau\left(p_i \right) \nonumber\\
&= \lambda^2 \sum_{i=1}^n \tau\left(p_i \right)\nonumber\\
&= \lambda^2 \tau\left(\bigvee_{i=1}^np_i \right), \quad (\text{again as $\sum_{k=1}^n p_k = \bigvee_{k=1}^n p_k$})
\end{align}
in which to get the last inequality we note that $p_is_i^2p_i \geq \frac{\lambda^2}{\alpha_i^2} p_i$. To show this, let $\xi \in \mathcal{H}$. Evidently, we may assume that $\|p_i\xi \| = 1$. It follows from the definition of $p_i$ that $p_i \xi \in \textbf{1}_{(\lambda, \infty)}\left(\alpha_i\,|s_i|\right)\left(\mathcal{H}\right)$. Applying the Holder--McCarthy inequality stating that $\langle T^2\eta, \eta\rangle \geq \langle T\eta, \eta \rangle^2$ for any positive operator $T$ and any unit vector $\eta$ (see \cite{MP}) we deduce that
\begin{align*}
\langle p_is_i^2p_i \xi, \xi \rangle = \langle s_i^2 p_i \xi, p_i \xi \rangle = \langle |s_i|^2 p_i \xi, p_i \xi \rangle > \frac{\lambda^2}{\alpha_i^2}.
\end{align*}
So $p_is_i^2p_i \geq \frac{\lambda^2}{\alpha_i^2} p_i$.\\
Letting $n \rightarrow \infty$ and using the normality of the trace, we get
\begin{align*}
\tau(p) \leq \frac{\sum_{n=1}^\infty\alpha_n^2{\rm var}(x_n)}{\lambda^2},
\end{align*}
where $p = \bigvee_{n=1}^\infty p_n$. Furthermore, by Lemma \ref{lem2} and Borel functional calculus
\begin{eqnarray*}
\left\| (1-p)\,\alpha_n|s_n|\,(1-p) \right\| \leq \left\| e_n\,\alpha_n|s_n|\,e_n \right\| = \left\| e_ne_{n-1}\,\alpha_n|s_n|\,e_{n-1}e_n \right\| \leq \lambda.
\end{eqnarray*}
\end{proof}
Putting $\alpha_k=1$ for all $k$ the noncommutative Haj\'ek--Renyi inequality reduces the noncommutative Kolmogorov inequality as follows.
\begin{corollary}[Noncommutative Chebyshev-type inequality]
Let $x_1, x_2, \ldots, x_n$ be successively independent self-adjoint random variables such that $\tau\left(x_k\right) = 0$ for all $k$. Then for any $\lambda>0$ there exists a projection $e$ such that
\begin{align*}
&\tau(e) \leq \frac{\sum_{k=1}^n{\rm var}(x_k)}{\lambda^2},\\
&\left\| (1-e)\,|s_n|\,(1-e) \right\| \leq \lambda \qquad (n \geq 1).
\end{align*}
\end{corollary}

The next result can be indeed considered as a (weak version of) Kolmogorov inequality for the maximal sum.

\begin{theorem}[Noncommutative Kolmogorov type inequality]\label{Kol2}
Let $x_1, x_2, \ldots, x_n$ be successively independent self-adjoint random variables in $\mathfrak{M}$ such that $\tau\left(x_k\right) = 0$ for all $k$. Then for any $\lambda > 0$ there exists a projection $e$ such that
\begin{align}\label{Kty}
1 - \frac{2(\lambda + \max_{1 \leq k \leq n} \|x_k\|)^2}{\sum_{k=1}^n {\rm var}(x_k)}\leq \tau(e)\leq \frac{\tau\left( \left(\sum_{k=1}^n x_k\right)^2 \right)}{\lambda^2}.
\end{align}
\end{theorem}
\begin{proof}
Set $c := \max_{1 \leq k \leq n} \|x_k\|$. Let $s_k := \sum_{i=1}^k x_i$ and $s_0 = 0$.\\
We set
\[e_0:=1, \quad e_k := e_{k-1}\textbf{1}_{[0, \lambda^2]}\left( e_{k-1}\, |s_k|^2\, e_{k-1} \right) \quad \text{and} \quad p_k := e_{k-1} - e_k\] for $1 \leq k \leq n$. It is clear that $(e_k)_k$ is a decreasing sequence of projections such that $e_k\, |s_k|^2\,e_k \leq \lambda^2 e_k$ for all $k$.\\
Moreover, $ p_j \wedge e_k=0$ for all $1 \leq k \leq n$ and all $1 \leq j \leq k$.
To see this, let $j \leq k$ and $\xi \in p_j(\mathcal{H}) \cap e_k(\mathcal{H})$ be a unit element. Then $\langle e_{j-1}\,|s_j|^2\,e_{j-1}\xi, \xi \rangle > \lambda^2$ and $\langle e_{j-1}\,|s_j|^2\,e_{j-1}\xi, \xi \rangle \leq \lambda^2$, since $\xi \in e_k(\mathcal{H}) \subseteq e_j(\mathcal{H})$, which gives us a contradiction.\\
In addition, $\vee_{j=1}^k p_j \leq \vee_{j=1}^k e_j^\perp = e_k^\perp$ since $p_j = e_{j-1} - e_j \leq e_j^\perp$ and $(e_j^\perp)_j$ is increasing.\\
Since $e_{k-1} = e_k + p_k$, we have
\begin{align}\label{equ1}
\tau\left(e_{k-1}\, |s_k|^2\,e_{k-1} \right)=\tau(|s_k|e_{k-1}|s_k|) = \tau(|s_k|(e_k + p_k)|s_k|)=\tau\left((e_k + p_k)\, |s_k|^2\,(e_k + p_k) \right).
\end{align}
Thus,
\begin{align*}
&\hspace{-0.5cm}\tau\left( \left| s_{k-1}e_{k-1} \right|^2 \right) + \tau\left( e_n\right){\rm var}(x_k)\\
&\leq \tau\left( \left| s_{k-1}e_{k-1} \right|^2 \right) + \tau\left( e_{k-1}\right){\rm var}(x_k) \qquad \qquad \qquad (\text{since $e_n \leq e_{k-1}$})\\
&= \tau\left( \left| s_{k-1}e_{k-1} \right|^2 \right) + \tau\left( e_{k-1}\right)\tau\left(x_k^2 \right) \qquad \qquad \qquad (\text{as $\tau\left( x_k\right)=0$})\\
&= \tau\left( \left| s_{k-1}e_{k-1} \right|^2 \right) + \tau\left(\left| x_ke_{k-1}\right|^2 \right)\\
& \qquad \qquad (\text{by the successive independence and the tracial property of $\tau$})\\
&= \tau\left(e_{k-1}s_{k-1}^2e_{k-1}\right) + \tau\left(e_{k-1}x_k^2e_{k-1}\right) + 2\tau\left(e_{k-1}s_{k-1}\right)\tau(x_k)\\
& \qquad\qquad\qquad\qquad\qquad\qquad\qquad\qquad\qquad \qquad (\text{as $\tau(x_k)=0$ for all $k$})\\
&=\tau\left(e_{k-1}s_{k-1}^2e_{k-1}\right) + \tau\left(e_{k-1}x_k^2e_{k-1}\right) + \tau\left(e_{k-1}s_{k-1}x_k\right) + \tau\left(s_{k-1}e_{k-1}x_k\right)
\end{align*}
by the successive independence. Thus,
\begin{align} \label{msm3}
&\tau\left( \left| s_{k-1}e_{k-1} \right|^2 \right) + \tau\left( e_n\right){\rm var}(x_k)\nonumber \\
&\leq \tau\left(e_{k-1}s_{k-1}^2e_{k-1}\right) + \tau\left(e_{k-1}x_k^2e_{k-1}\right) + \tau\left(e_{k-1}s_{k-1}x_ke_{k-1}\right) + \tau\left(e_{k-1}x_ks_{k-1}e_{k-1}\right) \nonumber\\
&= \tau\left(\left(e_{k-1}s_{k-1} + e_{k-1}x_k\right)\left(s_{k-1}e_{k-1} + x_ke_{k-1}\right) \right) \nonumber\\
&= \tau\left(\left|s_{k-1}e_{k-1} + x_ke_{k-1} \right|^2 \right) \nonumber\\
&= \tau \left(\left| s_ke_{k-1} \right|^2\right) \nonumber\\
&= \tau\left( e_{k-1}\,|s_k|^2\,e_{k-1}\right) \nonumber\\
&= \tau\left((e_k + p_k)\, |s_k|^2\,(e_k + p_k) \right) \qquad \qquad \qquad \qquad \qquad\qquad \qquad \qquad \qquad (\text{by \eqref{equ1}})\nonumber \\
&= \tau\left( |s_ke_k|^2 \right) + \tau\left( |s_kp_k |^2 \right) + \tau\left(e_k\,|s_k|^2\,p_k \right) + \tau\left(p_k\,|s_k|^2\,e_k \right) \nonumber\\
&= \tau\left( |s_ke_k|^2 \right) + \tau\left(|(s_{k-1}+ x_k)p_k |^2 \right) + \tau\left(|s_k|^2\,p_ke_k \right) + \tau\left(|s_k|^2\,e_kp_k \right) \nonumber\\
&= \tau\left( |s_ke_k|^2 \right) + \tau\left( |\left(s_{k-1}p_k + x_kp_k \right) |^2 \right)\qquad \qquad \qquad \qquad (\text{since $e_kp_k = p_ke_k =0$})\nonumber \\
&= \tau\left( |s_ke_k|^2 \right) + \tau\left(p_k \left(s_{k-1} + x_k\right)^2p_k \right) \nonumber \\
&= \tau\left( |s_ke_k|^2 \right) + \tau\left(p_k \left|s_{k-1} + x_k\right|^2p_k \right) \\
&\leq \tau\left( |s_ke_k|^2 \right) + 2\tau\left(p_k(\left|s_{k-1}\right|^2 + \left|x_k\right|^2)p_k \right) \nonumber\\
&\leq \tau\left( |s_ke_k|^2 \right) + 2\tau\left(p_k\,|s_{k-1}|^2\,p_k \right) + 2c^2\tau\left(p_k \right) \nonumber\\
&\leq \tau\left( |s_ke_k|^2 \right) + 2\tau\left(p_k \left(|s_{k-1}| + c\right)^2p_k \right)\nonumber
\end{align}

Thus
\begin{align}\label{in6}
&\tau\left( \left| s_{k-1}e_{k-1} \right|^2 \right) + \tau\left( e_n\right){\rm var}(x_k) \nonumber\\
&\leq \tau\left( |s_ke_k|^2 \right) + 2\tau\left(p_ks_{k-1}^2p_k \right) + 2c^2\tau\left(p_k \right) + 4c\tau\left(p_k|s_{k-1}|p_k \right)\qquad (\text{since $|x_k| \leq c$})\nonumber\\
&\leq \tau\left( |s_ke_k|^2 \right) + 2(\lambda + c)^2 \tau\left(p_k \right),
\end{align}
where to obtain the last inequality we remark that
\begin{eqnarray}
&& p_ks_{k-1}^2p_k \leq \lambda^2p_k \label{in7} \\
&& p_k\,|s_{k-1}|\,p_k \leq \lambda p_k.\label{in8}
\end{eqnarray}
To show \eqref{in7}, we may choose $\xi \in \mathcal{H}$ such that $\|p_k\xi\| = 1$. So
\begin{eqnarray*}
\langle p_ks_{k-1}^2p_k \xi, \xi \rangle = \langle |s_{k-1}|^2e_{k-2}p_k \xi, e_{k-2}p_k\xi \rangle \leq \lambda^2,
\end{eqnarray*}
as $p_k \xi \in e_{k-1}(\mathcal{H}) \subseteq e_{k-2}(\mathcal{H})$. The inequality \eqref{in8} follows from
\begin{eqnarray*}
\left(p_k\,|s_{k-1}|\,p_k \right)^2 = p_k\,|s_{k-1}|\,p_k\,|s_{k-1}|\,p_k \leq p_k\,s_{k-1}^2p_k \leq \lambda^2p_k \qquad (\text{by~} \eqref{in7}).
\end{eqnarray*}
by taking the square roots.\\
It follows from \eqref{in6} that
\begin{align*}
\tau\left( e_n\right){\rm var}(x_k) \leq \tau\left( |s_ke_k|^2 \right) - \tau\left( \left| s_{k-1}e_{k-1} \right|^2 \right) + 2(\lambda + c)^2 \tau\left(p_k \right).
\end{align*}
Now by summing and telescoping, we deduce that
\begin{eqnarray}\label{Kty2}
\tau\left( e_n\right)\sum_{k=1}^n {\rm var}(x_k) &\leq & \tau\left( |s_ne_n|^2 \right) + 2(\lambda + c)^2\, \tau\left(\bigvee_{k=1}^np_k \right)\nonumber\\
&\leq & \lambda^2 \tau\left(e_n \right) + 2(\lambda + c)^2\, \tau\left(1 - e_n \right)\nonumber\\
&\leq & 2(\lambda + c)^2,
\end{eqnarray}
from which, by setting $e = 1 - e_n$, we get (\ref{Kty}).\\
It should be noted that the second inequality can be derived from the following.
\begin{eqnarray*}
\langle e_n s_n^2 e_n \xi, \xi \rangle &=& \langle |s_n|^2 e_n\, \xi, e_n\xi \rangle\\
&= & \langle e_{n-1}\,|s_n|^2\, e_{n-1}e_n \xi, e_n \xi \rangle \quad \left(\text{as $e_{n-1}e_n = e_n$}\right)\\
&\leq & \lambda^2 \langle e_n \xi, e_n \xi \rangle
\end{eqnarray*}
when $\|e_n\xi\|=1$, since $e_n \xi \in \textbf{1}_{[0, \lambda^2]}\left( e_{n-1} |s_n| e_{n-1} \right)\left(\mathcal{H} \right)$. Hence $|s_ne_n|^2 = e_n s_n^2 e_n \leq \lambda^2 e_n$, and so $\tau\left( |s_ne_n|^2 \right) \leq \lambda^2 \tau\left(e_n \right)$.\\
Furthermore, we by \eqref{GH1} have
\begin{eqnarray*}
\tau(e) = \tau\left(1_{(\lambda^2, \infty)}\left( e_{n-1}\, |s_n|^2\, e_{n-1} \right) \right) \leq \frac{\tau\left( e_{n-1}\, |s_n|^2\, e_{n-1} \right)}{\lambda^2} \leq \frac{\tau\left(|s_n|^2 \right)}{\lambda^2}.
\end{eqnarray*}
\end{proof}
\begin{remark}\label{rem2}
We may derive the noncommutative version of Kolmogorov-type inequality with the same bound if the hypothesis $s_ks_{k-1} = s_{k-1}s_k$ satisfies for all $k$, where $s_k = \sum_{i=1}^k x_i$. In fact, from \eqref{msm3} we have
\begin{align*}
\tau\left( \left| s_{k-1}e_{k-1} \right|^2 \right) + \tau\left( e_n\right){\rm var}(x_k)&\leq \tau\left( |s_ke_k|^2 \right) + \tau\left(p_k \left|s_{k-1} + x_k\right|^2p_k \right) \\
&\leq \tau\left( |s_ke_k|^2 \right) + \tau\left(p_k \left(|s_{k-1}| + |x_k|\right)^2p_k \right)
\end{align*}
since by passing to the commutative algebra generated by two commuting self-adjoint elements $x_k$ and $s_{k-1}$, we have
$|s_{k-1} + x_k |^2\leq (|s_{k-1}| + |x_k|)^2$.\\
Continuing the previous argument we deduce
\begin{align*}
\tau(e_n) \leq \frac{(\lambda + c)^2}{\sum_{k=1}^n {\rm var}(x_k)}.
\end{align*}
\end{remark}
In the next result we deduce the commutative version of the Kolmogorov type inequality.
\begin{corollary}\cite[Theorem 3.1.7]{G}
Let $X_1, X_2, \ldots, X_n$ be independent random variables with the mean $0$ and $\sup_k |X_k| \leq C$ for some constant $C > 0$. Then
\begin{align*}
\mathbb{P}\left(\max_{1 \leq k \leq n} \left| S_k \right| > t \right) \geq 1 - \frac{(t + C)^2}{\sum_{k=1}^n {\rm Var}\, X_k}
\end{align*}
for any $t > 0$.
\end{corollary}

In \cite{BAT}, a sufficient condition is stated for the convergence of a independent sequence of random variables.
In the next result, we show that the condition is necessary under some conditions.
\begin{corollary}
If $\left(x_n\right)_{n \in \mathbb{N}}$ is a successively independent sequence of self-adjoint random variables in $\mathcal{M}$ such that $\tau\left(x_k\right) = 0$ for all $k$ and $\sum_{n=1}^\infty x_n$ is convergent, then $\sum_{n=1}^\infty {\rm var}(x_n)$ is convergent too.
\end{corollary}
\begin{proof}
Suppose that for some $c > 0$ $\| x_n \| \leq c$ for all $n$. Set $s_n := \sum_{k=1}^n x_k$. To achieve a contradiction, assume that $\sum_{n\in \mathbb{N}} {\rm var}(x_n)=\infty$. Let $\epsilon >0$ be given and $n \in \mathbb{N}$ be arbitrary. For every $m \in \mathbb{N}$, applying Theorem \ref{Kol2} to the finite sequence $x_{n+1}, x_{n+2}, \ldots, x_{n+m}$, there exists a projection $e_{n,m}$ such that
\begin{align}\label{msm2}
1 \geq \tau(e_{n,m}^\perp) \geq 1 - \frac{2\epsilon^2 + 2c^2}{\sum_{k=n+1}^m {\rm var}(x_k)},
\end{align}
where $e_{n,m} = \textbf{1}_{[0, \epsilon^2]}\left( e_{n, m-1}\, |s_{n+m} - s_n|^2\, e_{n, m-1} \right)$ with $e_{n, 0} = 1$. Our construction shows that $(e_{n,m})_m$ is a decreasing sequence of projections. Letting $m \rightarrow \infty$ in \eqref{msm2} and noticing that $(e_{n,m}^\perp)_m$ is increasing, we deduce from the normality of $\tau$ that $\tau\left( \bigvee_{m \in \mathbb{N}} e_{n,m}^\perp \right)=\lim_m\tau(e_{n,m}^\perp)= 1=\tau(1)$. Therefore
\begin{align}\label{sad1}
\bigvee_{m \in \mathbb{N}} e_{n,m}^\perp = 1
\end{align}
for all $n$.\\
As $(s_n)_n$ is Cauchy, there is $N \in \mathbb{N}$ such that
\begin{align*}
\| (s_{N+m} - s_N)^2 \| = \| s_{N+m} - s_N\|^2 \leq \epsilon^2
\end{align*}
for all $m \in \mathbb{N}$. Let $ \xi \in e_{N,m}^\perp(\mathcal{H})$ be a unit vector. Hence, by definition of the spectral projection $e_{N,m}^\perp$ we have $\langle |s_{N+m} - s_N|^2\, e_{N,m-1}\xi, e_{N,m-1}\xi \rangle > \epsilon^2$. On the other hand, from $\| (s_{N+m} - s_N)^2 \| \leq \epsilon^2$ and using the Cauchy--Schwarz inequality, we get
\begin{align*}
\langle |s_{N+m} - s_N|^2\, e_{N,m-1}\xi, e_{N,m-1}\xi \rangle \leq \epsilon^2,
\end{align*}
which is impossible. Hence $e_{N,m}^\perp(\mathcal{H}) = 0$. Thus, $ \bigvee_{m \in \mathbb{N}} e_{N,m}^\perp = 0$ contradicting \eqref{sad1}.
\end{proof}

\section{Noncommutative Etemadi inequality}

To establish a noncommutative version of the Etemadi inequality we need the following simple lemma.
\begin{lemma}\label{lem1}
Let $p$, $q$ and $r$ be projections in $\mathfrak{M}$ satisfying $p \leq q$ and $p \leq r$. Then $\tau(p) \leq \tau\left(qr\right)$.\\
In particular, for arbitrary projections $p, q$ it holds that $\tau\left(p \wedge q\right) \leq \tau(pq)$.
\end{lemma}
\begin{proof}
\begin{eqnarray*}
\tau(p) = \tau\left(p^{\frac{1}{2}}pp^{\frac{1}{2}}\right) \leq \tau\left( p^{\frac{1}{2}}rp^{\frac{1}{2}}\right) = \tau\left( r^{\frac{1}{2}}pr^{\frac{1}{2}}\right) \leq \tau\left(qr\right).
\end{eqnarray*}
\end{proof}

\begin{theorem}[Noncommutative Etemadi inequality]\label{Etem}
Let $x_1, x_2, \ldots, x_n$ be weakly fully independent self-adjoint random variables in $\mathfrak{M}$ such that $s_ks_n =s_n s_k$, where $s_k = \sum_{j=1}^k x_j$ for all $1 \leq k \leq n$. Then for every $\lambda > 0$ there exists a projection $p$ in $\mathfrak{M}$ such that
\begin{align*}
\tau (p)\leq 2\tau\left(\textbf{{\rm \textbf{1}}}_{[\lambda, \infty)}\left(|s_n|\right)\right) + \max_{1 \leq k \leq n} \tau\left(\textbf{{\rm \textbf{1}}}_{[\lambda, \infty)}\left(|s_k|\right)\right)\leq 3\max_{1 \leq k \leq n} \tau\left( \textbf{{\rm \textbf{1}}}_{[\lambda, \infty)} \left( |s_k|\right) \right).
\end{align*}
\begin{eqnarray*}
\tau(1-p) \leq \tau\left(\textbf{{\rm \textbf{1}}}_{[0, 3\lambda)}(|x_1|) \right).
\end{eqnarray*}
\end{theorem}
\begin{proof}
As in the proof of Noncommutative Haj\'ek--Renyi inequality, we inductively define Cuculescu's projections
\begin{align*}
e_k = \textbf{1}_{[0, 3 \lambda)}\left( e_{k-1} |s_k| e_{k-1} \right)e_{k-1} \quad (1\leq k\leq n)
\end{align*}
with $e_0 = 1$.
Note that $e_{k-1}$ commutes with $e_{k-1} |s_k| e_{k-1}$ and hence with $\textbf{1}_{[0, 3 \lambda)}\left( e_{k-1} |s_k| e_{k-1} \right)$. Thus $e_k \in W^*\left(x_1, x_2, \ldots, x_k\right)$ $(1\leq k\leq n)$ is also a projection and as before the sequence $\{e_k\}$ is decreasing.
Now define the projections $p_k$ in $W^*\left(x_1, x_2, \ldots, x_k\right)$ by
\begin{align*}
p_k = e_{k-1} \textbf{1}_{[3\lambda, \infty)}\left(e_{k-1}\,|s_k|\,e_{k-1}\right), ~ k=1, 2, \ldots, n.
\end{align*}
Then
\begin{align}\label{eq3}
p_k p_j =0 ~ (k \neq j).
\end{align}
To see \eqref{eq3}, without loss of generality, assume that $k < j$. We have
\begin{eqnarray*}
p_k p_j =& e_{k-1} \textbf{1}_{[3\lambda, \infty)}\left(e_{k-1}\,|s_k|\,e_{k-1}\right)e_{j-1} \textbf{1}_{[3\lambda, \infty)}\left(e_{j-1}\,|s_j|\,e_{j-1}\right)\\
= & \textbf{1}_{[3\lambda, \infty)}\left(e_{k-1}\,|s_k|\,e_{k-1}\right)e_{j-1} \textbf{1}_{[3\lambda, \infty)}\left(e_{j-1}\,|s_j|\,e_{j-1}\right) \quad (\text{as $e_{j-1} \leq e_{k-1}$}).
\end{eqnarray*}
Now we show that the last term is zero. For simplicity set $q_k = \textbf{1}_{[3\lambda, \infty)}\left(e_{k-1}\,|s_k|\,e_{k-1}\right)$ and $q_j =\textbf{1}_{[3\lambda, \infty)}\left(e_{j-1}\,|s_j|\,e_{j-1}\right)$. It follows from $e_{j-1} \leq e_k$ and definition of $e_k$ that
\begin{eqnarray*}
\| q_ke_{j-1}q_j\|^2 = \| q_je_{j-1}q_k\|^2 = \| q_ke_{j-1}q_je_{j-1}q_k\| \leq \| q_ke_{j-1}q_k \| \leq \| q_ke_kq_k \| = \| 0 q_k \| = 0,
\end{eqnarray*}
where the first inequality is due to $q_j \leq 1$. Set $p = \vee_{k=1}^n p_k = \sum_{k=1}^n p_k$.
By an inductive argument, $s_n$ commutes with all $e_k$ and so with all $p_k$, or equivalently, $s_n\left(p_k(\mathcal{H})\right) \subseteq p_k(\mathcal{H})$. Hence $s_n\left(p(\mathcal{H})\right) \subseteq p(\mathcal{H})$. Therefore $ps_np=s_np$, whence $ps_n=ps_np=s_np$.\\
We have
\begin{align*}
\tau (p) &= \tau\left(p \wedge \textbf{1}_{(\lambda, \infty)}\left(|s_n|\right) + p \wedge \textbf{1}_{[0, \lambda]}\left(|s_n|\right)\right)\\
& \qquad\qquad\qquad (\text{by} ~ p = p\wedge q + p \wedge q^\perp ~ \text{whenever}~ pq = qp)\\
&\leq \tau\left(\textbf{1}_{(\lambda, \infty)}\left(|s_n|\right)\right) + \tau\left(\left(\vee_{k=1}^n p_k\right) \wedge \textbf{1}_{[0, \lambda]}\left(|s_n|\right)\right) \\
&\leq \tau\left(\textbf{1}_{(\lambda, \infty)}\left(|s_n|\right)\right) + \tau\left(\vee_{k=1}^n \left(p_k \wedge \textbf{1}_{[0, \lambda]}\left(|s_n|\right)\right)\right)\\
&\qquad\qquad\qquad\qquad\qquad\qquad\qquad\qquad (\text{since~} s_np_k = p_ks_n)\\
&\leq \tau\left(\textbf{1}_{(\lambda, \infty)}\left(|s_n|\right)\right) + \sum_{k=1}^n \tau\left( p_k \wedge \textbf{1}_{[0, \lambda]}\left(|s_n|\right)\right)\quad (\text{by \eqref{NE}})\\
&\leq \tau\left(\textbf{1}_{(\lambda, \infty)}\left(|s_n|\right)\right) + \sum_{k=1}^n \tau\left( p_k \wedge \textbf{1}_{[2\lambda, \infty)}\left(|s_n - s_k|\right)\right).
\end{align*}
To reach the last inequality, consider the commutative algebra $\mathcal{N}$ generated by $s_k, s_n$. If $\xi \in p_k(\mathcal{H}) \cap \textbf{1}_{[0, \lambda]}\left(|s_n|\right)(\mathcal{H})$ is a unit vector, then
\begin{eqnarray*}
\langle |s_k|\xi, \xi \rangle \geq 3\lambda \quad \text{and} \quad \langle |s_n|\xi, \xi \rangle \leq \lambda.
\end{eqnarray*}
Therefore,
\begin{eqnarray*}
\langle |s_n - s_k|\xi, \xi \rangle \geq \langle |s_k|\xi, \xi \rangle - \langle |s_n|\xi, \xi \rangle \geq 2\lambda.
\end{eqnarray*}
Hence $p_k \wedge \textbf{1}_{[0, \lambda]}\left(|s_n|\right) \leq p_k \wedge \textbf{1}_{[2\lambda, \infty)}\left(|s_n - s_k|\right)$. Therefore,
\begin{align*}
\tau (p) &\leq \tau\left(\textbf{1}_{(\lambda, \infty)}\left(|s_n|\right)\right) + \sum_{k=1}^n \tau\left( p_k \textbf{1}_{[2\lambda, \infty)}\left(|s_n - s_k|\right)\right) \quad \text{( Lemma \ref{lem1})}\\
&= \tau\left(\textbf{1}_{(\lambda, \infty)}\left(|s_n|\right)\right) + \sum_{k=1}^n \tau\left( p_k\right) \tau\left(\textbf{1}_{[2\lambda, \infty)}\left(|s_n - s_k|\right)\right)
\end{align*}
by the weakly full independence property of $W^*(x_1, \ldots, x_k)$ and $W^*(x_{k+1}, \ldots, x_n)$.\\
Therefore,
\begin{align}\label{rema1}
\tau (p)&\leq \tau\left(\textbf{1}_{(\lambda, \infty)}\left(|s_n|\right)\right) + \max_{1 \leq k \leq n} \tau\left(\textbf{1}_{[2\lambda, \infty)}\left(|s_n - s_k|\right)\right) \tau\left(\vee_{k=1}^n p_k\right) \quad (\text{by \eqref{eq3}})\\
&\leq \tau\left(\textbf{1}_{(\lambda, \infty)}\left(|s_n|\right)\right) + \max_{1 \leq k \leq n} \tau\left(\textbf{1}_{[2\lambda, \infty)}\left(|s_n - s_k|\right)\right)\nonumber\\
&\leq \tau\left(\textbf{1}_{(\lambda, \infty)}\left(|s_n|\right)\right) + \max_{1 \leq k \leq n} \left\{\tau\left(\textbf{1}_{[\lambda, \infty)}\left(|s_n|\right)\right) + \tau\left(\textbf{1}_{[\lambda, \infty)}\left(|s_k|\right)\right)\right\}\nonumber\\
& ~\qquad \qquad (\text{by considering the commutative algebra generated by $s_k, s_n$})\nonumber\\
&=2\tau\left(\textbf{1}_{(\lambda, \infty)}\left(|s_n|\right)\right) + \max_{1 \leq k \leq n} \tau\left(\textbf{1}_{[\lambda, \infty)}\left(|s_k|\right)\right)\nonumber\\
&\leq 3\max_{1 \leq k \leq n} \tau\left(\textbf{1}_{[\lambda, \infty)}\left(|s_k|\right)\right).\nonumber
\end{align}
Moreover,
\begin{eqnarray*}
\tau(1-p) = \tau\left(\wedge_k p_k^\perp \right) \leq \tau(p_1^\perp) = \tau\left(1_{[0, 3\lambda)}(|x_1|) \right).
\end{eqnarray*}
\end{proof}
\begin{remark}
Notice that the projection $p$ in the previous theorem is nonzero provided that $\tau\left(\textbf{1}_{[3\lambda, \infty)}\left(|s_k| \right)\right) \neq 0$ for some $1 \leq k \leq n$. In fact, if this condition occurs and $p=0$, then $p_k=0$ for all $k$. In particular, $p_1 = \textbf{1}_{[3\lambda, \infty)}\left(|s_1|\right) =0$, and hence $e_1=\textbf{1}_{[0, 3\lambda)}\left(|s_1|\right) = 1$. Hence $e_2 = \textbf{1}_{[0, 3\lambda)}\left(|s_2|\right)$. Inductively, we can deduce that $e_k=1$, and so $\textbf{1}_{[3\lambda, \infty)}\left(|s_k| \right) =0$, which gives rise contradiction.
\end{remark}
\begin{remark}
We should remark that $n$ is fixed and we only require that $s_k$ commutes with $s_n$ for all $k$. Thus our condition is weaker than the commutativity condition $x_ix_j = x_jx_i$ for all $i, j$. For example, for $n=4$, set
$$x_1 = \begin{pmatrix}
2 & i\\
-i & 2
\end{pmatrix},
x_2 = \begin{pmatrix}
1 & -i\\
i & 2
\end{pmatrix},
x_3 = \begin{pmatrix}
2 & 1+i\\
1-i & 3
\end{pmatrix} {\rm ~and~}
x_4 = \begin{pmatrix}
3 & -1-i\\
-1+i & 1
\end{pmatrix}.$$ Then $x_1$, $x_1+x_2$ and $x_1+x_2+x_3$ commute with $x_1+x_2+x_3+x_4 = \begin{pmatrix}
8 & 0\\
0 & 8
\end{pmatrix}$, but not all $x_1, x_2, x_3, x_4$ commute with each other.
\end{remark}

\begin{remark}
From inequality \eqref{rema1}, we conclude that
\begin{align*}
\left(1 - \max_{1 \leq k \leq n} \tau\left(\textbf{1}_{(\frac{\lambda}{2}, \infty)}\left(|s_n - s_k|\right)\right) \right)\tau (p) \leq \tau\left(\textbf{1}_{(\lambda, \infty)}\left(|s_n|\right)\right)
\end{align*}
and hence
\begin{align*}
\tau (p) \leq \frac{\tau\left(\textbf{1}_{(\lambda, \infty)}\left(|s_n|\right)\right)}{\left(1 - \max_{1 \leq k \leq n} \tau\left(\textbf{1}_{(\frac{\lambda}{2}, \infty)}\left(|s_n - s_k|\right)\right) \right)}
\end{align*}
provided that $\max_{1 \leq k \leq n} \tau\left(\textbf{1}_{(\frac{\lambda}{2}, \infty)}\left(|s_n - s_k|\right)\right) \neq 1$.
\end{remark}

\begin{corollary}
Let $X_1, X_2, \ldots, X_n$ be independent random variables. If $t>0$ is an arbitrary positive real number, then
\begin{align*}
\mathbb{P}\left(\max_{1 \leq k \leq n} \left|\sum_{i=1}^k X_i\right| \geq 3t \right) \leq 3 \max_{1 \leq k \leq n} \mathbb{P}\left(\left|\sum_{i=1}^k X_i\right| \geq t \right).
\end{align*}
\end{corollary}
\begin{proof}
The projection $e_j$ in the proof of noncommutative Etemadi inequality above corresponds to the characteristic function of the subset $E_j = \{\chi_{E_{j-1}}|S_j|<3\lambda\} \cap E_{j-1}$ with $E_0=\Omega$.\\
It is easy to check that
\begin{align*}
\bigcup_{k=1}^nP_k=\left\{\max|S_k|\geq 3\lambda \right\},
\end{align*}
whenever
\begin{align*}
P_k=E_{k-1} \cap\left\{|S_k|\geq 3\lambda \right\}.
\end{align*}
\end{proof}


\begin{thebibliography}{99}

\bibitem{BAT} C.J.K. Batty, \textit{The strong law of large numbers for states and traces of a $W\sp{\ast} $-algebra}, Z. Wahrsch. Verw. Gebiete \textbf{48} (1979), no. 2, 177--191.

\bibitem{C} I. Cuculescu, \textit{Martingales on von Neumann algebras}, J. Multivariate Anal. \textbf{1} (1971), 17--27.

\bibitem{E} N. Etemadi, \textit{On some classical results in probability theory}, Sankhy Ser A \textbf{47} (1985), no. 2.

\bibitem{G} A. Gut, \textit{Probability: a graduate course}, Second edition, Springer Texts in Statistics, Springer, New York, 2013.

\bibitem{HR} J. Haj\'ek and A. Renyi, \textit{Generalization of an inequality of Kolmogorov}, Acta Math. Acad. Sci. Hungar. \textbf{6} (1955), no. 3-4. 281�283.

\bibitem{JZ} M. Junge and Q. Zeng, \textit{Noncommutative Bennett and Rosenthal inequalities}, Ann. Probab. \textbf{41} (2013), no. 6, 4287--4316.

\bibitem{JX} M. Junge and Q. Xu, \textit{Noncommutative Burkholder/Rosenthal inequalities. II. Applications}, Israel J. Math. \textbf{167} (2008), 227--282.

\bibitem{MP} Mond, J. Pecaric, \textit{On Jenesen's inequality for operator convex functions}, Houston J. Math. \textbf{21} (1995), 739--754.

\bibitem{N} E. Nelson, \textit{Notes on non-commutative integration}, J. Funct. Anal., \textbf{15} (1974), 103--116.

\bibitem{R} N. Randrianantoanina, \textit{A weak type inequality for noncommutative martingales and applications}, Proc. London. Math. Soc. \textbf{91}, (2005) no. 3, 509--544.

\bibitem{SM1} Gh. Sadeghi and M.S. Moslehian, \textit{Noncommutative martingale concentration inequalities}, Illinois J. Math. \textbf{58} (2014), no. 2, 561--575.

\bibitem{SM2} Gh. Sadeghi and M. S. Moslehian, \textit{Inequalities for sums of random variables in noncommutative probability spaces}, Rocky Mount. J. Math. \textbf{46} (2016), no. 1, 309--323.

\bibitem{TMS} A. Talebi, M. S. Moslehian and Gh. Sadeghi, \textit{Noncommutative Blackwell--Ross martingale inequality}, Infin. Dimens. Anal. Quantum Probab. Relat. Top. (to appear).

\bibitem{Xu2} Q. Xu, \textit{Operator spaces and noncommutative $L_p$}, Lectures in the Summer School on Banach spaces and Operator spaces, Nankai University China, 2007.

\end{thebibliography}
\end{document}